\documentclass[11pt]{amsart}
\usepackage{amsmath}
\usepackage{amssymb}
\usepackage{tabularx}
\usepackage{enumerate}
\usepackage{graphicx}
\usepackage{texdraw}

\usepackage{mathrsfs}

\usepackage{amsfonts,amssymb,amsmath}
\usepackage{epsfig}

\makeatletter
\@addtoreset{equation}{section}

\makeatother
\newtheorem{thm}{Theorem}[section]
\newtheorem{lem}[thm]{Lemma}
\newtheorem{cor}[thm]{Corollary}
\newtheorem{prop}[thm]{Proposition}
\newtheorem{remark}[thm]{Remark}

\newcommand{\R}{\mathbb{R}}

\begin{document}
\title[Asymptotic Behavior of Solutions of a Reaction Diffusion Equation]
{Asymptotic Behavior of Solutions of a Reaction Diffusion Equation with Free Boundary
Conditions$^\S$}
 \thanks{$\S$ This research was partly supported by the NSFC (No. 11271285). }
\author[J. Cai, B. Lou and M. Zhou]{Jingjing Cai$^\dag$,  Bendong Lou$^\ddag$
and Maolin Zhou$^\sharp$}
\thanks{$\dag$ School of Mathematics and Physics, Shanghai University of Electric Power,
Shanghai 200090, China}
\thanks{$\ddag$ Department of Mathematics, Tongji University, Shanghai 200092, China.}
\thanks{$^\sharp$ Graduate School of Mathematical Sciences, The University of Tokyo,
Tokyo 153-8914, Japan}
\thanks{{\bf Emails:} {\sf cjjing1983@163.com} (J. Cai), {\sf
blou@tongji.edu.cn} (B. Lou), {\sf zhouml@ms.u-tokyo.ac.jp} (M. Zhou) }
%\date{Feb. 18, 2011}

\begin{abstract} We study a nonlinear diffusion equation of the
form $u_t=u_{xx}+f(u)\ (x\in [g(t),h(t)])$ with free boundary conditions
$g'(t)=-u_x(t,g(t))+\alpha$ and $h'(t)=-u_x(t,g(t))-\alpha$ for some
$\alpha>0$. Such problems may be used to describe the spreading of a
biological or chemical species, with the free boundaries representing
the expanding fronts. When $\alpha=0$, the problem was recently
investigated by \cite{DuLin, DuLou}. In this paper we consider the case
$\alpha>0$. In this case shrinking (i.e. $h(t)-g(t)\to 0$) may happen,
which is quite different from the case $\alpha=0$.
Moreover, we show that, under certain conditions on $f$, shrinking is
equivalent to vanishing (i.e. $u\to 0$), both of them happen as $t$
tends to some finite time. On the other hand, every bounded and positive
time-global solution converges to a nonzero stationary solution as $t\to \infty$.
As applications, we consider monostable and bistable types of nonlinearities,
and obtain a complete description on the asymptotic behavior of the solutions.
\end{abstract}

 \maketitle

\section{Introduction}

In this paper we consider the following problem
\begin{equation}\label{p}
\left\{
\begin{array}{ll}
 u_t = u_{xx} + f(u), &  g(t)< x<h(t),\ t>0,\\
 u(t,g(t))= u(t,h(t))=0 , &  t>0,\\
g'(t)= - u_x(t, g(t)) +\alpha, & t>0,\\
 h'(t) = - u_x (t, h(t)) - \alpha, & t>0,\\
-g(0)=h(0)= h_0,\ \ u(0,x) =u_0 (x),& -h_0\leq x \leq h_0,
\end{array}
\right.
\end{equation}
where $x=g(t)$ and $x=h(t)$ are moving boundaries to be
determined together with $u(t,x)$, $\alpha >0$ is a given constant,
\begin{equation}\label{cond1}
f: [0,\infty) \to \R \mbox{ is a } C^1 \mbox{ function} ,\quad f(0)=0.
\end{equation}
The initial function $u_0$ belongs to  $ \mathscr {X}(h_0)$ for some
$h_0>0$, where
\begin{equation}\label{def:X}
\begin{array}{ll}
\mathscr {X}(h_0):= \Big\{ \phi \in C^2 ([-h_0,h_0]): & \phi(-h_0)=
\phi (h_0)=0,\; \phi'(-h_0)>0,\\ & \phi'(h_0)<0,\;
 \phi(x) >0 \ \mbox{in } (-h_0,h_0)\;\Big\}.
\end{array}
\end{equation}
For any given $h_0>0$ and $u_0 \in \mathscr {X}(h_0)$, by a
(classical) solution of \eqref{p} on the time-interval $[0,T]$ we
mean a triple $(u(t,x), g(t), h(t))$ belonging to  $ C^{1,2}(G_T)
\times C^1 ([0,T])\times C^1 ([0,T])$, such that all the identities
in \eqref{p} are satisfied pointwisely, where
\[
G_T:=\big\{(t,x): t\in (0,T],\; x\in [g(t), h(t)]\big\}.
\]

Recently, problem \eqref{p} with $\alpha =0$ was studied in \cite{DuLin, DuLou}, etc.
When $f(u)$ is of monostable or bistable type of nonlinearity, the problem may be used to
describe the spreading of a new biological or chemical species. The free
boundaries $x=g(t)$ and $x=h(t)$ represent the spreading fronts of
the species whose density is represented by $u(t,x)$. The free boundary condition
with $\alpha=0$ is a Stefan one. Its biological meaning can be found in
\cite{BDK, DuLin}. In \cite{DuLin, DuLou} the authors studied the asymptotic behavior of
the solutions of \eqref{p} (with $\alpha=0$) and proved that
any bounded time-global solution converges to a stationary one as $t\to \infty$.
Among others, their results show that vanishing (i.e. $u\to 0$) may happen even for an
equation with logistic nonlinearity provided the initial data is small enough. Such a result show that
problem \eqref{p} with $\alpha =0$ has advantages comparing with the Cauchy problems.
(The Cauchy problem for an equation with logistic nonlinearity has
{\it hair-trigger} effect, that is, any positive solution converges to a positive constant,
cf. \cite{AW1, AW2, DuLin, DuLou}).
In the last two years, \cite{DG1,DG2, DMW} also studied the corresponding problems of
\eqref{p} with $\alpha=0$ in higher dimension spaces.

In this paper we consider the free boundary condition with a real number $\alpha>0$.
We use this parameter to denote a spreading resistant force representing the
reluctance of the individuals of the species to move away from the population region.
Intuitively, the presence of $\alpha>0$ makes the solution more difficult to spread
than the case where $\alpha=0$. Indeed, $h'(t)>0$ only if $ u_x(t,h(t)) <-\alpha$, that is,
the solution spreads only if the pressure at the boundary is big enough.

The main purpose of this paper is to study the asymptotic behavior of bounded
solutions of \eqref{p}. As we will see below, for a solution $(u,g,h)$ of \eqref{p},
either {\it spreading} (i.e. $h(t),-g(t)\to \infty$ and $u$ converges to a positive constant),
or {\it vanishing} (i.e. $u \to 0$), or {\it shrinking} (i.e. $h(t) - g(t) \to 0$),
or {\it transition} (i.e. $u$ converges to a stationary solution with compact support) happens.
Comparing with the results in \cite{DuLin, DuLou}, the shrinking phenomena is a
new one since it does not happen in case $\alpha =0$.

A simple variation of the arguments in \cite{DuLin} shows
that, for any $h_0>0$ and $u_0\in  \mathscr {X}(h_0)$, \eqref{p} has a unique solution
defined on some maximal time interval $(0, T_*)$ with $T_*\in (0, \infty]$.
No matter $T_*<\infty$ or $T_*=\infty$ we will show that $g(t)$ and $h(t)$
have limits:

\begin{prop}\label{prop:exist limits}
Let $(u,g,h)$ be a solution of \eqref{p} on some maximal time interval $[0, T_*)$
with $T_*\in (0, \infty]$. Then the following limits exist:
\begin{equation}\label{exist limits}
g_*:= \lim\limits_{t\to T_*} g(t) \quad \mbox{and} \quad
h_*:= \lim\limits_{t\to T_*} h(t).
\end{equation}
\end{prop}

\noindent
This proposition is proved by the fact that $h(t)$, as well as $g(t)$, does not move across
any fixed point for infinitely many times (see details in subsection 2.5).
We write $I_*:= [g_*, h_*]$ in what follows. In particular, when $T_* =\infty$ we also write $g_*, h_*$ and $I_*$
as $g_\infty, h_\infty$ and $I_\infty$, respectively.

When $T_* =\infty$, the solution is a time-global one and so we can study its asymptotic behavior.
On the other hand, $T_*$ may be a finite number for some reasons like
blow up, shrinking or vanishing, etc.
We are not concerned with the blow up phenomena in this paper, so we impose
the following condition
\begin{equation}\label{cond2}
f(u)\leq Ku \ \mbox{ for all } u\geq 0 \mbox{ and some } K>0
\end{equation}
to exclude the possibility that $u$ blows up in finite time.

Recall that we introduced $\alpha>0$ in the free boundary conditions.
Hence, the properties $h'(t)>0$ and $g'(t)<0$ in case $\alpha=0$
(as shown in \cite{DuLin, DuLou}) are no longer necessarily to be true.
Instead, the domain $I(t):= [g(t),h(t)]$ may shrink, even, to a point.
For some $\widetilde{T}\in (0,\infty]$,
we say that {\it shrinking happens, or the interval $[g(t),h(t)]$ shrinks as $t\to \widetilde{T}$} if

\bigskip
\noindent
\underline{\it Shrinking}: \ \ $h(t) -g(t) >0 $ for $t\in [0, \widetilde{T})$ and
$\lim_{t\to \widetilde{T}}h(t) = \lim_{t\to \widetilde{T}}g(t) \in \R$.
\bigskip

\noindent
This is a new phenomena which never happens when $\alpha=0$ (cf. \cite{DuLin, DuLou}).
A related phenomena is {\it vanishing}: for some $\widetilde{T}\in (0,\infty]$,
we say that {\it $u$ vanishes, or vanishing happens as $t\to \widetilde{T}$} if
$$
\underline{Vanishing}: \ \
 \left\{
 \begin{array}{l}
 (a)\ \tilde{g}:= \lim_{t\to \widetilde{T}}g(t) <  \tilde{h}:=\lim_{t\to \widetilde{T}}h(t),
  \mbox{ and } u(t,\cdot) \to 0 \mbox{ as } t\to \widetilde{T} \\
  \ \ \ \mbox{\ \ locally uniformly in }  (\tilde{g}, \tilde{h}), \ or\\
  (b)\ \mbox{shrinking happens as } t\to \widetilde{T} \mbox{ and }
  \lim_{t\to \widetilde{T}} \|u(t,\cdot)\|_{L^\infty ([g(t),h(t)])} =0.
\end{array}
 \right.
 \hskip 30mm
$$

\noindent
As a definition we list two cases for vanishing, we will show later that
case (a) indeed does not occur (see Theorem \ref{thm:vanishing}, Lemma \ref{lem:T_*<infty}
and Remark \ref{remark:not (a)}).

For the sake of clarity, when shrinking or vanishing happens as $t\to \widetilde{T}$
for some finite time $\widetilde{T}$, in this paper we always say that the maximal existence interval
of the solution is $[0,\widetilde{T})$.

On shrinking and vanishing phenomena we have the following questions:
Whether vanishing or shrinking really happens for some solutions? If one of them
happens, does the other one happen at the same time?
Do they happen in finite or infinite time?
The first question is answered in Proposition \ref{prop:vanish cond}, where
we give some sufficient conditions which guarantee that $u$ vanishes, for example,
\begin{equation}\label{cond3}
\sup\limits_{u\geq 0} F(u) >0,\quad \mbox{ where } F(u):= \int_0^u f(s) ds£¬
\end{equation}
and $\alpha >\alpha_0 := [2\sup_{u\geq 0} F(u)]^{1/2}$.
The second and the third  questions are answered by the following theorem.

\begin{thm}\label{thm:vanishing}
Assume \eqref{cond1}. Let $(u,g,h)$ be a solution of problem \eqref{p} on some maximal time interval
$(0, T_*)$. Then the following statements are equivalent:
\begin{itemize}
\item[(i)] $T_* <\infty$;
\item[(ii)] vanishing happens as $t\to T_*$;
\item[(iii)] shrinking happens as $t\to T_*$.
\end{itemize}
\end{thm}

\noindent
This theorem is proved at the end of of section 3.
For each time-global solution $(u,g,h)$, Theorem \ref{thm:vanishing} implies that
vanishing and shrinking do not happen. So we are interested in the asymptotic
behavior for time-global solutions and we have the following general convergence theorem,
which is an analogue of
Theorem 1.1 in \cite{DM} and Theorem 1.1 in \cite{DuLou}.

\begin{thm}\label{thm:convergence}
Assume \eqref{cond1}. Let $(u, g,h)$ be a solution of \eqref{p}
that is defined for all $t>0$. Assume that $u(t,x)$ is bounded, namely
\[
u(t,x)\leq C \mbox{ for all $t>0$, \ $x\in [g(t), h(t)]$ and some
$C>0$}.
\]
Then $I_\infty := (g_\infty, h_\infty)$ is either a finite interval or $I_\infty =\R^1$.

Moreover, if $I_\infty$ is a finite interval, then $h_\infty = g_\infty + 2 \ell$ for some
$\ell >0$ (cf. \eqref{def-ell}) and
$$
\lim_{t\to\infty} u(t, \cdot)= V_\alpha (\cdot) \mbox{ locally uniformly in } (g_\infty, h_\infty),
$$
where $V_\alpha$ is the unique solution of
\begin{equation} \label{ellip_0}
   \left\{
   \begin{array}{l}
   v'' +f(v)=0,\quad x\in (g_\infty, h_\infty),\\
   v(g_\infty) =v(h_\infty)=0,\ \ v' (g_\infty) =\alpha,\ v'(h_\infty) = - \alpha.
   \end{array}
   \right.
 \end{equation}

If $(g_\infty, h_\infty)=\R^1$ then either $\lim_{t\to\infty} u(t,x)$ is a positive constant
solution of
\begin{equation}\label{ellip_1}
 v_{xx}+f(v)=0,\; x\in \R^1,
 \end{equation}
 or
 \begin{equation}\label{to V(gamma)}
 u(t,x)- V(x+\gamma(t))\to 0 \mbox{ as } t\to\infty,
 \end{equation}
 where $\gamma: [0,\infty)\to [-h_0,h_0]$ is a continuous function,
 $V$ is an evenly decreasing positive solution of \eqref{ellip_1}.
\end{thm}

\begin{remark}\label{V_infty>B}
The conclusion in \eqref{to V(gamma)} is possible only if \eqref{cond3} holds,
$\alpha \leq \alpha_0$ and if
\begin{equation}\label{V>B}
 V_\infty := \lim\limits_{x\to \infty} V(x) \geq B := \min\{ \bar{v}: \alpha^2 = 2F(\bar{v} )\}.
\end{equation}
\end{remark}

As applications, we study two typical types of nonlinearities: \vskip 6pt
\begin{center}
(f$_M$)\ \ monostable case, \ \ \ (f$_B$)\ \ bistable case.
\end{center}
In the monostable case (f$_M$), we assume that $f$ is $C^1$ and it
satisfies
\begin{equation}\label{mono}
f(0)=f(1)=0, \quad f'(0)>0, \quad  f'(1)<0,\quad (1-u)f(u) >0 \ \mbox{for } u>0, u\not= 1.
\end{equation}
One example is $f(u) =u(1-u)$. In the bistable case (f$_B$), we assume that $f$ is $C^1$ and it
satisfies
\begin{equation}\label{bi}
f(0)=f(\theta)= f(1)=0, \quad f(u) \left\{
\begin{array}{l}
<0 \ \ \mbox{in } (0,\theta),\\
>0\ \  \mbox{in } (\theta, 1),\\
< 0\ \ \mbox{in } (1,\infty)
\end{array} \right.
\end{equation}
for some $\theta\in (0,1)$,  $f'(0)<0$, $f'(1)<0$ and
\begin{equation}\label{unbalance}
F(1)= \int_0^1 f(s) ds >0.
\end{equation}
A typical example is $f(u) = u(u-\theta)(1-u)$ with $\theta \in (0, \frac{1}{2})$.
Note that when $f$ is of (f$_B$) type and when $\alpha >0$, the condition \eqref{V>B}
in Remark \ref{V_infty>B} is not satisfied for the unique ground state $V$, and so the
convergence in \eqref{to V(gamma)} does not occur.

Clearly  \eqref{cond1}, \eqref{cond2} and \eqref{cond3} are satisfied if $f$ is of
{\rm (f$_M$)}, or of {\rm (f$_B$)} type.
The next theorem gives a rather complete description for the asymptotic behavior
of the solutions of \eqref{p} with monostable or bistable type of nonlinearity.

\begin{thm}\label{thm:mono-bi}
 Assume that $f$ is of {\rm (f$_M$)}, or {\rm (f$_B$)} type and $0<\alpha <  \sqrt{2F(1)} $.
 Let $(u,g,h)$ be a solution of \eqref{p} on some maximal interval $[0,T_*)$. Then either

{\rm (i) Spreading:} $T_*=\infty$, $(g_\infty, h_\infty)=\R^1$ and
\[
\lim_{t\to\infty}u(t,x)=1 \mbox{ locally uniformly in $\R^1$},
\]
or

{\rm (ii) Vanishing:} $T_* <\infty$, $\lim_{t\to T_*}g(t) =
\lim_{t\to T_*} h(t)  \in [-h_0,h_0]$ and
\[
\lim_{t\to T_*}\max_{g(t)\leq x\leq h(t)} u(t,x)=0,
\]
or

{\rm (iii) Transition:} $T_* =\infty$, $h_\infty = g_\infty +2\ell$ and
\[
\lim_{t\to\infty} u(t,\cdot) = V_\alpha (x) \mbox{ locally uniformly in } (g_\infty, h_\infty),
\]
where $V_\alpha$ is the unique positive solution to \eqref{ellip_0},
\begin{equation}\label{def-ell}
\ell := \int_0^B \frac{dr}{\sqrt{\alpha^2 - 2 F(r) }}
\quad \mbox{with } B\in (0,1) \mbox{ given by } \alpha^2 = 2 F(B) .
\end{equation}

Moreover, if $u_0=\sigma \phi$ with $\phi\in \mathscr {X}(h_0)$,
then there exists $\sigma^* = \sigma^* (h_0, \phi) \in (0,\infty]$
such that vanishing happens when $ 0< \sigma < \sigma^*$,
spreading happens when $ \sigma > \sigma^*$, and transition happens
when $\sigma=\sigma^*$.
\end{thm}

Theorem \ref{thm:mono-bi} is an analogue of Theorem 1.3 in \cite{DM} (for Cauchy problems) and
Theorems 1.2  and 1.3 in \cite{DuLou} (for \eqref{p} with $\alpha =0$), but
 they are different. First, Theorem \ref{thm:mono-bi} (ii)
 (together with Theorem \ref{thm:vanishing}) means that
 shrinking is possible, and shrinking happens at the same time as
 vanishing. Second, transition in Theorem \ref{thm:mono-bi} (iii) means that $u$
 converges to a stationary solution with compact support.
 From ecological point of view, this means that a species can survive
 forever in a bounded area without changing its population density.
 On the other hand, for the problems studied in \cite{DM, DuLou}, transition
 does not happen in the problems with monostable $f$, and it does happen
 in the problems with bistable $f$ but $u$ converges to the ground state defined on the whole space.
In \cite{Cai}, using a different approach the author also studied problem
\eqref{p} with $f(u)=u(1-u)$ and obtained similar results as in Theorem \ref{thm:mono-bi}.

In \cite{DM, DuLou}, the authors also studied the equation with combustion type of
nonlinearity. From a mathematical point of view, one of course can study the problem
\eqref{p} with combustion type of $f$. We remark that similar conclusions as in
Theorem \ref{thm:mono-bi} hold for this kind of $f$ (we omit the details in this paper).

Finally we remark that when spreading happens (Theorem \ref{thm:mono-bi} (i)),
the asymptotic spreading speed can be studied in the same way as in
\cite{DuLin, DuLou, DMZ}, etc. Indeed, the spreading speed is determined by the following problem
\begin{equation}\label{prop-profile}
\left\{
  \begin{array}{l}
  q_{zz} - c q_z + f(q) =0\ \ \mbox{ for }  z\in (0,\infty),\\
  q(0)=0, \; q_z(0) = c+\alpha,\; q(\infty)=1,\; q(z)>0 \mbox{ for } z>0.
  \end{array}
  \right.
\end{equation}

\begin{prop}\label{prop:speed} Assume that $f$ is of {\rm (f$_M$)}, or {\rm
(f$_B$)} type. If $ 0< \alpha < \sqrt{2F(1)}$, then \eqref{prop-profile} has a
unique solution $(c,q)=(c^*, q^*)$ with $c^*>0$.  Moreover, when spreading happens, we have
\[
 \begin{array}{c}
\lim_{t\to\infty} [h(t) -c^* t -H] = 0,\quad \lim\limits_{t\to \infty} h'(t)=c^*,\\
\lim_{t\to\infty} [g(t) +c^* t -G] = 0,\quad \lim\limits_{t\to \infty} g'(t)= -c^*
\end{array}
\]
for some $H,G\in \R$.
\end{prop}

We omit the proof of this proposition since it is similar as that in \cite{DMZ}.

The rest of the paper is organized as follows. In section 2, we
present some basic results which are fundamental for this research.
In section 3 we discuss the vanishing phenomena, and give necessary conditions
for vanishing and for shrinking. We also prove Theorem 1.2 at the end of this section.
In section 4 we prove Theorem 1.3. In section 5 we give some
sufficient conditions for vanishing. In section 6 we prove Theorem 1.5.
%Finally we give a deduction for the free boundary condition in the Appendix.

\section{Some Basic Results}\label{sec:basic}

In this section we give some basic results which will be
used later in the paper. The results here are for general $f$ which
satisfies \eqref{cond1}.

\subsection{Time-local existence}
The following local existence result can be proved by the same
argument as in \cite{DuLin}.

\begin{thm}
\label{thm:local} Suppose that \eqref{cond1} holds. For any given
$u_0\in \mathscr {X}(h_0)$ and any $\nu\in (0,1)$, there is a
$T>0$ such that Problem \eqref{p} admits a unique solution
$$(u, g, h)\in C^{(1+\nu)/2, 1+\nu}(\overline{G}_{T})\times C^{1+\nu/2}([0,T])\times C^{1+\nu/2}([0,T]);$$
moreover,
\begin{eqnarray}
\|u\|_{C^{ (1+\nu)/2, 1+\nu}(\overline{G}_{T})}+\|g\|_{C^{1+\nu/2}([0,T])}+\|h\|_{C^{1+\nu/2}([0,T])}\leq C, %\overline
\label{b12}
\end{eqnarray}
where $G_{T}=\{(t,x)\in \R^2: x\in [g(t), h(t)], t\in (0,T]\}$, $C$
and $T$ only depend on $h_0$, $\nu$ and $\|u_0\|_{C^{2}([-h_0,
h_0])}$.
\end{thm}

\begin{remark}\rm
As in \cite{DuLin}, by the Schauder estimates applied to the
equivalent fixed boundary problem used in the proof, we have
additional regularity for $u$, namely,  $u\in C^{1+\nu/2, 2+\nu}(G_T)$.
\end{remark}

\subsection{Comparison principles}
\begin{lem}
\label{lem:comp1} Suppose that \eqref{cond1} holds, $T\in
(0,\infty)$, $\overline g, \overline h\in C^1([0,T])$, $\overline
u\in C(\overline D_T)\cap C^{1,2}(D_T)$ with $D_T=\{(t,x)\in\R^2:
0<t\leq T, \overline g(t)<x<\overline h(t)\}$, and
\begin{eqnarray*}
\left\{
\begin{array}{lll}
\overline u_{t} \geq \overline u_{xx} +f(\overline u),\; & 0<t \leq T,\
\overline g(t)<x<\overline h(t), \\
\overline u= 0,\quad \overline g'(t)\leq - \overline u_x + \alpha,\quad &
0<t \leq T, \ x=\overline g(t),\\
\overline u= 0,\quad \overline h'(t)\geq - \overline u_x -\alpha,\quad
&0<t \leq T, \ x=\overline h(t).
\end{array} \right.
\end{eqnarray*}
If $[-h_0, h_0]\subseteq [\overline g(0), \overline h(0)]$,
$u_0(x)\leq \overline u(0,x)$ in $[-h_0,h_0]$, and if
$(u,g, h)$ is a solution of \eqref{p}, then
\[
g(t)\geq \overline g(t),\; h(t)\leq\overline h(t) \mbox{ in } (0,
T], \quad u(x,t)\leq \overline u(x,t) \mbox{ for } t\in (0, T] \mbox{ and }
x\in (g(t), h(t)).
\]
\end{lem}

\begin{lem}
\label{lem:comp2} Suppose that  \eqref{cond1} holds, $T\in
(0,\infty)$, $\overline g,\, \overline h\in C^1([0,T])$, $\overline
u\in C(\overline D_T)\cap C^{1,2}(D_T)$ with $D_T=\{(t,x)\in\R^2:
0<t\leq T, \overline g(t)<x<\overline h(t)\}$, and
\begin{eqnarray*}
\left\{
\begin{array}{lll}
\overline u_{t}\geq  \overline u_{xx}+ f(\overline u),\; &0<t \leq T,\ \overline g(t)<x<\overline h(t), \\
\overline u\geq u, &0<t \leq T, \ x= \overline g(t),\\
\overline u= 0,\quad \overline h'(t)\geq - \overline u_x -\alpha,\quad
&0<t \leq T, \ x=\overline h(t),
\end{array} \right.
\end{eqnarray*}
with $\overline g(t)\geq g(t)$ in $[0,T]$, $h_0\leq \overline
h(0)$, $u_0(x)\leq \overline u(0,x)$ in $[\overline g(0),h_0]$,
where  $(u,g, h)$ is a solution of \eqref{p}. Then
\[
\mbox{ $h(t)\leq\overline h(t)$ in $(0, T]$,\quad $u(x,t)\leq
\overline u(x,t)$ for $t\in (0, T]$\ and $ \overline g(t)<x< h(t)$.}
\]
\end{lem}

These comparison principles are the same as those in \cite{DuLou}.
The proof of Lemma \ref{lem:comp1} is identical to that of Lemma 5.7
in \cite{DuLin}, and a minor modification of this proof yields Lemma
\ref{lem:comp2}.

The function $\overline u$, or the triple
$(\overline u,\overline g,\overline h)$, in Lemmas \ref{lem:comp1}
and \ref{lem:comp2} is often called an {\it upper solution} of \eqref{p}.
A {\it lower solution} can be defined analogously by reversing all the
inequalities. We also have corresponding comparison results for
lower solutions in each case.

\subsection{A priori estimates for $h'$ and $g'$}

\begin{lem}
\label{lem:bound-general} Suppose that \eqref{cond1} holds,
$(u,g,h)$ is a solution of \eqref{p} defined for $t\in [0,T_0)$ for
some $T_0\in (0,\infty)$, and there exists $C_1>0$ such that
\[
0< u(t,x)\leq C_1 \mbox{ for } t\in [0,T_0) \mbox{ and } x\in (g(t),
h(t)).
\]
Then there exists $C_2$ depending on $C_1$ but independent of $T_0$
such that
\[
-\alpha < -g'(t), h'(t) \leq C_2\; \mbox{ for } t\in (0, T_0).
\]

Moreover, the solution can be extended to some  interval $(0, T)$
with $T>T_0$ as long as $\inf_{0<t<T_0} [h(t)-g(t)]>0$.
\end{lem}

\begin{proof}
We only give the estimates for $h'$, the estimates for $g'$ is proved similarly.

By the maximum principle and Hopf lemma we have $u_x(t,h(t))< 0$, and so $h'(t)=-u_x(t,h(t))-\alpha>-\alpha$.
Next we give the upper bound of $h'$. Following the proof of Lemma 2.2 in \cite{DuLin}
we construct a function of the form
\begin{equation}\label{def-U}
U(t,x)=C_1 \big[2M(h(t)-x)-M^2 (h(t)-x)^2\big]
\end{equation}
over the region
\[
Q:= \{(t,x):0<t<T_0, \ \max\{ h(t)-M^{-1}, g(t)\} <x<h(t)\},
\]
where
$$
M:= \max\left\{ \frac{\alpha +\sqrt{\alpha^2 +2K_1}}{2}, \frac{4\|u_0\|_{C^1([-h_0,h_0])}}{3C_1} \right\}
$$
and $K_1 := \max_{0\leq u\leq C_1} |f'(u)|$.

Clearly $0\leq U\leq C_1$ in $Q$. By the definitions of $U, M$ and $K_1$ we have
$$
U_t - U_{xx} -f(U) \geq C_1 [2 M^2 - 2 \alpha M - K_1] \geq 0\quad \mbox{in }  Q.
$$
Moreover,
$$
U(t,h(t))= u(t,h(t))=0 \mbox{ for } t\in (0,T_0),
$$
$$
u_0 (x) \leq U(0,x)  \mbox{ for } x\in [h_0 -M^{-1}, h_0] \cap [-h_0, h_0],
$$
$$
U (t, h(t)- {M}^{-1} ) =C_1 \geq u (t, h(t)- M^{-1} )\quad  \mbox{ when }
h(t)-g(t) \geq  M^{-1}
$$
and
$$
U(t, g(t)) > 0 = u (t, g(t))\quad  \mbox{ when } h(t)-g(t) <  M^{-1}.
$$
Therefore, $u(t,x)\leq U(t,x)$ in $Q$ by the comparison principle Lemma \ref{lem:comp2}. Thus
$$
h'(t)=-u_x(t,h(t))-\alpha \leq - U_x(t,h(t))-\alpha \leq C_2 := 2MC_1 -\alpha.
$$

Now we assume $\rho := \inf_{0<t<T_0} [h(t)-g(t)]>0$ and to prove that
the solution $(u,g,h)$ can be extended to some interval $(0, T)$
with $T>T_0$.  From the above estimates we have
\[
  -g(t), h(t)\in [h_0 -\alpha t, h_0+C_2t],\quad -g'(t), h'(t)\in (-\alpha , C_2]\quad
  \mbox{ for } t\in [0, T_0).
\]
We now fix $\delta\in (0, T_0)$. By standard $L^p$ estimates, the
Sobolev embedding theorem, and the H\"{o}lder estimates for
parabolic equations, we can find $C_3>0$ depending only on
$\delta$, $T_0$, $C_1$, $C_2$ such that $||u(t,\cdot)||_{C^{2}([g(t), h(t)])}\leq
C_3$ for $t\in [\delta, T_0)$. It then follows from the proof of
Theorem~\ref{thm:local} (cf. \cite{DuLin}) that there exists a $\tau >0$ depending on
$C_3$, $C_2$, $C_1$ and $\rho$ but not on $t$ such that the solution of problem
\eqref{p} with initial time $t\in [\delta, T_0)$ can be extended uniquely to the time
$t+ 2\tau $. In particular, if we start from time $T_0-\tau$, then we
can extend the solution to time $T_0+\tau $.
\end{proof}

This lemma implies that the solution of \eqref{p} can be
extended as long as $u$ remains bounded, $u\not\equiv 0$ and $h(t)-g(t)>0$.
So we have the following theorem.

\begin{thm}\label{thm:global} Suppose that \eqref{cond1} holds.
Then the problem \eqref{p} has a unique solution defined on some maximal interval $[0, T_*)$
with $T_*\in (0,\infty]$.

Assume further that \eqref{cond2} holds. Then $T_*<\infty$ if and only if
shrinking or vanishing happens as $t\to T_*$.
\end{thm}

\subsection{Stationary solutions}

In this subsection we study stationary solutions of $u_t=u_{xx} +f(u)$,
More precisely, consider the following problem:
\begin{equation}\label{ellip-p}
\left\{
\begin{array}{l}
v^{\prime\prime}+f(v)=0,\quad x>0,\\
v(0)=0,\quad v'(0) =\alpha,
\end{array}
\right.
\end{equation}
where $\alpha >0$ is the constant in \eqref{p}.
Multiplying the equation by $2v'$ and integrating it on $[0,x]$ we have
$$
(v')^2 = \alpha^2 - 2F(v),\quad x>0,
$$
where $F(v):= \int_0^v f(s) ds$. Thus
$$
\int_0^v \frac{dr}{\sqrt{\alpha^2 - 2F(r)}} =x,\ \quad x>0.
$$

We want to classify all the solutions of \eqref{ellip-p}.
When \eqref{cond3} holds, we define $\alpha_0>0$ by
\begin{equation}\label{def-alpha_0}
\alpha^2_0 := 2  \sup_{v>0} F(v) .
\end{equation}
Clearly,
$$
S:= \{v>0: \alpha^2 = 2F(v)\}
 \left\{
  \begin{array}{ll}
   \not= \emptyset ,& \mbox{if } \alpha <\alpha_0;\\
   \not= \emptyset, & \mbox{if } \alpha=\alpha_0\mbox{ and } \alpha_0^2  = 2F(\bar{v})
     \mbox{ for some } \bar{v}>0;\\
    =\emptyset, &  \mbox{if } \alpha > \alpha_0;\\
    =\emptyset, & \mbox{if } \alpha =\alpha_0 \mbox { and } \alpha^2_0 > 2F(v)
    \mbox{ for all } v>0,\\
    =\emptyset, & \mbox{if } \eqref{cond3} \mbox{ does not hold}.
  \end{array}
  \right.
$$
Set
\begin{equation}\label{def-ell-1}
B := \left\{
 \begin{array}{ll}
 \min S, & \mbox{when } S \not=\emptyset, \\
 \infty, & \mbox{when } S=\emptyset,
  \end{array}
 \right.
\quad \mbox{ and } \quad
\ell := \int_0^B \frac{dr} { \sqrt{\alpha^2 - 2F(r)}} .
\end{equation}
By a simple analysis for all of these cases we have the following result.

\begin{lem}\label{lem:stationary-alpha}
When \eqref{cond3} holds, the solutions of \eqref{ellip-p} are divided into the following cases.
\begin{itemize}
 \item[(i)] $\alpha< \alpha_0$ and $\ell < \infty$ holds. In this case the unique
 solution (denoted by $V_\alpha(x)$) of \eqref{ellip-p} is defined on $[0, 2 \ell]$,
 $V'_\alpha (x)>0$ in $[0, \ell)$, $V_\alpha (\ell) =B$ and $V_\alpha$ is symmetric
 with respect to $x= \ell$;
 \item[(ii)] $\alpha < \alpha_0$ and $\ell = \infty$, or
 $\alpha^2 =\alpha_0^2  = 2F(\bar{v})$ for some $\bar{v}$ (denote the minimum of
     such $\bar{v}$ by $B$). In these cases the unique solution (denoted by
     $\widetilde{V}_\alpha(x)$) of \eqref{ellip-p}  is defined in $[0,\infty)$,
 $\widetilde{V}'_\alpha (x)>0$ and $\widetilde{V}_\alpha (x)\to B$ as $x\to \infty$;
 \item[(iii)] $\alpha > \alpha_0$, or $\alpha^2 =\alpha^2_0 > 2F(v)$ for all $v>0$.
 In these cases the unique solution (denoted by $\widehat{V}_\alpha(x)$) of \eqref{ellip-p}
 is defined in $[0, \ell)$, $\widehat{V}'_\alpha (x)>0$ and $\widehat{V}_\alpha (x)
 \to \infty$ as $x\to \ell$;
\end{itemize}

When \eqref{cond3} does not hold, the solution is like $\widehat{V}_\alpha(x)$ in case {\rm (iii)}.
\end{lem}

It is easily seen that, if $V_\alpha (x)$ (resp. $\widetilde{V}_\alpha (x)$,
$\widehat{V}_\alpha (x)$) is a solution of \eqref{ellip-p}, then for any $b\in\R$
$V_\alpha (\pm x+b)$ (resp. $\widetilde{V}_\alpha (\pm x +b)$,
$\widehat{V}_\alpha (\pm x+b)$) are also stationary solutions of $u_t = u_{xx} +f(u)$.

\subsection{Existence of the limits of $h(t)$ and $g(t)$}

\begin{lem}\label{lem:finite-osc}
Let $(u,h,g)$ be a solution of \eqref{p} defined on some maximal existence
interval $[0,T_*)$.  Then for any $b\in \R$,
$h(t)-b$ changes sign at most finite many times. The same is true for $g(t)-b$.
\end{lem}

\begin{proof}
We only consider the case that Problem \eqref{ellip-p} has a solution
$V_\alpha (x)$ on the compact interval $[0,2\ell]$ as in Lemma \ref{lem:stationary-alpha} (i),
and prove the lemma for $h(t)-b$. Other cases can be proved similarly.

(1) We first consider the case where $h_0 \leq b-2\ell$. In this case if
$h(t)$ moves (rightward) across $b-2\ell$ at some time (denote $t_1$ the first of such times),
then just after $t_1$, the function $\eta(t,x):= u(t,x)-V_\alpha (x-b+2\ell)$
has exactly one zero $z(t)$ on $[b-2\ell, h(t)]$.
By the maximum principle, the number of zeros of $\eta(t,\cdot)$ remains $1$ until
one of the following three cases happens.

(i) $h(t)$ shrinks back and crosses $b-2\ell$ again.
Then the situation becomes the same as in the very beginning.

(ii) $g(t)$ moves (rightward) across $b-2\ell$ at time $t_2$ while $h(t)$ remains
in $(b-2\ell, b]$ in time interval $(t_1, t_2]$. In this case
$g'(t_2) \geq 0$ and we have the following claim:
\vskip 6pt
\noindent
\underline{{\it Claim} 1}: $z(t)$ moves to $b-2\ell$ as $t\to t_2$, $u(t,x)\leq V_\alpha (x-b+2\ell)$
for $t>t_2$.
\vskip 6pt
\noindent
If $b-2\ell < z(t_2)$, then $\eta(t_2, \cdot)$ has exactly two zeros $b-2\ell$ and $z(t_2)$.
Consider $\eta(t,x)$ in the domain
$\{(t,x): b- 2\ell < x < z(t), t_1 < t \leq t_2\}$. By the maximum principle we have
$$
\eta(t_2, x) > 0 \mbox{ for } x\in (b-2\ell , z(t_2)),\quad
\eta(t_2, b-2\ell )=0 \quad \mbox {and } \quad \eta_x (t_2, b-2\ell) > 0.
$$
The last inequality implies that
$$
g'(t_2) = -u_x (t_2, b-2\ell) +\alpha < -V'_\alpha (0) +\alpha =0,
$$
contradicting $g'(t_2)\geq 0$. This proves the first part of Claim 1.
Hence $z(t_2)=b-2\ell$ and $u(t_2, x) \leq V_\alpha(x-b+2\ell)$ on
$[b-2\ell, h(t_2)] \subset [b-2\ell, b]$. By the comparison principle
(Lemma \ref{lem:comp1}) we have $u(t, x) \leq V_\alpha(x-b+2\ell)$ for all $x\in (g(t), h(t))$
and $t>t_2$. This proves Claim 1.

Consequently, $h(t)\leq b$ for all $t>0$ in case (ii).

(iii) $h(t)$ moves (rightward) across $b$ at time $t_3$ while $g(t)<b-2\ell$ for all
$t\in [t_1, t_3]$. In this case $h'(t_3)\geq 0$ and we have the following claim:
\vskip 6pt
\noindent
\underline{{\it Claim} 2}: $z(t)$ moves to $b$ as $t\to t_3$.
\vskip 6pt
\noindent
Otherwise, $z(t_3) < b$ and so we can use the maximum principle in $\{(t,x):
z(t)< x < h(t), t_1 < t \leq t_3\}$ to conclude that $\eta_x (t_3, b) > 0$.
This implies that $h'(t_3) = -u_x(t_3, b) - \alpha < -V'_\alpha (2\ell) -\alpha =0$,
contradicting $h'(t_3)\geq 0$. This proves Claim 2. Consequently,
$z(t_3) =b$ and $\eta (t_3,x) >0$ in $[b-2\ell, b)$. Therefore,
$u(t,x) >V_\alpha (x-b+2\ell)$ for all $t>t_3$. Hence $h(t)-b$ changes sign
only once till time $t_3$.

(2) Next we consider the case where $b\in (-h_0, h_0)$.

(i) If $-h_0 \leq b-2\ell$ and $u_0(x) \geq V_\alpha (x-b+2\ell)$ on $[b-2\ell, b]$,
then $u(t,x) > V_\alpha (x-b+2\ell)$ for all $t>0$ and $x\in [b-2\ell, b]$ by comparison
principle, and so $h(t)>b$ for all $t>0$.

(ii) If $u_0(x) - V_\alpha (x-b+2\ell) <0$ at some points,
then by \cite[Lemma 2.3]{DM} or \cite[Theorems C and D]{A}, for any $t>0$, the function
$\eta(t,x) := u(t,x) - V_\alpha (x-b+2\ell)$ has finite number of zeros on its domain
$J(t) := [x_1 (t), x_2(t)]$, where
$$
x_1(t) := \max\{ g(t), b-2\ell\},\quad  x_2(t) := \min\{ h(t), b\}.
$$
Denote the largest zero of $\eta(t,x)$ on $J(t)$ by $\tilde{z}(t)$. Clearly, $\tilde{z}(t) < x_2(t) = b$
for small $t>0$.

If $h(t)$ moves (leftward) across $b$ at time $t_4>0$ (with $h(t)>b$ for $t\in (0,t_4)$),
then $h'(t_4)\leq 0$ and we have the following claim.
\vskip 6pt
\noindent
\underline{{\it Claim} 3}: $\tilde{z}(t)$ moves to $b$ as $t\to t_4$, and this zero disappear just after $t_4$.
\vskip 6pt
\noindent
The former part of this claim is proved in a similar way as proving Claim 1.
So $\tilde{z}(t_4) =b$ and $\eta (t_4,x)<0$ just on the left side of $b$ (since the zeros of $\eta(t_4,\cdot)$
are discrete). Using Hopf lemma we have $\eta_x(t_4, b) >0$ and so
$h'(t_4) = -u_x(t_4, b) - \alpha < 0$. Thus $h(t)<b$ for $t>t_4$ and $t-t_4$ small. This proves Claim 3.

Claim 3 implies that once $h(t)$ moves across $b$, the number of zeros of $\eta(t,\cdot)$
decreases strictly. Consequently, $h(t)-b$ can not change sign infinite many times.

(3) Other cases including $b\in [h_0, h_0+2\ell]$ and $b \leq -h_0$
can be studied similarly.
\end{proof}

\noindent
{\it Proof of Proposition \ref{prop:exist limits}}. The conclusions follow from
the previous lemma and its proof. \hfill $\Box$

\bigskip

By the proof of the previous lemma we also have the following result.

\begin{cor}\label{cor:h-g}
Let $(u,h,g)$ be a solution of \eqref{p} defined on some maximal
existence interval $[0,T_*)$. If \eqref{ellip-p} has solution $V_\alpha$ with compact
support and if $h(t_1) > h_0 +2\ell$ at time $t_1$, then
$u(t,x)>V_\alpha (x-h_0)$ for all $x\in [h_0 , h_0+2\ell]$ and $t>t_1$.
\end{cor}

\noindent
By Theorem \ref{thm:convergence}, we see that under the assumptions of this corollary,
$u(t,\cdot)$ converges as $t\to \infty$ to a positive solution of
$v''+f(v)=0$ locally uniformly in $\R$, and hence $h(t)\to \infty$ and $g(t)\to -\infty$.

\subsection{Estimates of $h(t)+g(t)$ and monotonicity of $u$}
\begin{lem}\label{lem:center}
Suppose that $(u,g, h)$ is a solution of \eqref{p} on $[0,T_*)$.
Then
\begin{equation}\label{g+h}
-2h_0<g(t)+h(t)<2h_0 \; \mbox{ for all } t\in [0,T_*).
\end{equation}
\end{lem}

This lemma can be proved by the maximum principle in a similar way as in
\cite[Lemma 2.8]{DuLou}.

\begin{lem}\label{lem:monotonicity}
{\rm (i)} Assume $g(T)=-h_0$ and $g(t)<-h_0\ (t>T)$. Then $u_x(t,x)>0$ for all $g(t)\leq x < -h_0$ and $t>T$.
{\rm (ii)} Assume $h(T)=h_0$ and $g(t)>h_0\ (t>T)$. Then $u_x(t,x)<0$ for all $h_0 < x\leq h(t)$ and $t>T$.
\end{lem}

\begin{proof}
We only prove (i) since (ii) is proved similarly. For this purpose,
we need to prove $u_x(\tilde{t}, \tilde{x}) >0$ for any given
$\tilde{t} >T$ and any $\tilde{x}\in [g(\tilde{t}), -h_0)$.

If $\tilde{x}=g(\tilde{t})$, then $u_x(\tilde{t},\tilde{x}) >0$ by Hopf lemma.
In what follows we assume $\tilde{x} \in (g(\tilde{t}), -h_0)$.
Then, there exists $t_1 \in (T, \tilde{t})$ such that $g(t_1)=\tilde{x}$ and $g(t)< \tilde{x}$ for
$t\in (t_1, \tilde{t}]$.  Consider
\[
z(t,x):=u(t,x)-u(t, 2\tilde{x}-x)
\]
over $G :=\{(t,x): t_1 <t \leq \tilde{t}, g(t)< x< \tilde{x}\}$. We have
\[
z_t=z_{xx}+c(t,x)z \mbox{ in } G, \quad (c \mbox{ is a bounded function}),
\]
\[
z(t, g(t))<0 \mbox{ and }  z(t, \tilde{x})=0 \mbox{ for } t_1<t\leq \tilde{t}.
\]
Hence we can apply the strong maximum principle and the Hopf lemma
to deduce
\[
z(t,x)<0 \mbox{ in } G,\quad z_x(t, \tilde{x})>0 \mbox{ for } t_1<t\leq \tilde{t}.
\]
In particular, we have $z_x(\tilde{t},\tilde{x})=2u_x(\tilde{t},\tilde{x})>0$.
\end{proof}

\section{Vanishing phenomena}\label{sec:vanishing}
In this section we assume that $(u,g,h)$ is a solution of \eqref{p} on its maximal existence interval $[0,T_*)$.

\subsection{Uniform convergence for vanishing case}
By Proposition \ref{prop:exist limits}, $g_*:= \lim_{t\to T_*} g(t)$ and
$h_*:=\lim_{t\to T_*} h(t)$ exist and $g_*\leq h_*$.
We now show that, when $g_* <h_*$, vanishing can happen in a unform topology.

\begin{lem}\label{lem:uniform vanish}
Assume $u(t,x)\leq C$ for all $0\leq t<T_*$ and $x\in [g(t),h(t)]$.
If $g_* < h_*$ and if $\lim_{t\to T_*}u(t,\cdot) =0$
locally uniformly in $I_* :=(g_*, h_*)$, then
$\|u(t,\cdot)\|_{L^\infty ([g(t),h(t)])} \to 0$ as $t\to T_*$.
\end{lem}

\begin{proof}
We first consider the case where $T_* =\infty$ and $h_* =\infty$, $g_*=-\infty$.
In this case, Lemma \ref{lem:monotonicity} implies that $u(t,x)$ is decreasing in $x>h_0$ and
increasing in $x<-h_0$ for large $t$. Hence
$\|u(t,\cdot)\|_{L^\infty ([-h_0,h_0])} = \|u(t,\cdot)\|_{L^\infty ([g(t),h(t)])}$
for large $t$. By our assumption, $\lim_{t\to \infty}u(t,\cdot) =0$
uniformly on $x\in [-h_0,h_0]$. Hence $\|u(t,\cdot)\|_{L^\infty ([g(t),h(t)])} \to 0$ as $t\to \infty$.

Next we consider the case where $T_* =\infty$, $-\infty< g_* =g_\infty< h_*=h_\infty <\infty$.
In the same way as in the proof of Lemma \ref{lem:bound-general} we construct a function
\[
U(t,x)=C \big[2M(h(t)-x)-M^2 (h(t)-x)^2\big]
\]
over the region $Q:= \{(t,x):t>0, \ \max\{ h(t)-M^{-1}, g(t)\} <x<h(t)\}$, where
$$
M:= \max\Big\{ \frac{\alpha +\sqrt{\alpha^2 +2K_1}}{2}, \frac{4\|u_0\|_{C^1([-h_0,h_0])}}{3C} \Big\}
$$
and $K_1 : =\sup_{0\leq u\leq C} |f'(u)|$. As in Lemma \ref{lem:bound-general} we have
$u(t,x)\leq U(t,x)$ in $Q$.

For any small $\varepsilon>0$, there exists $\delta>0$ small with
$$
\delta < \min\Big\{  \frac{1}{M},\ \frac{h_\infty -g_\infty}{4},\ \frac{\varepsilon}{4MC} \Big\}
$$
such that
\begin{equation}\label{delta skip}
U(t,x) \leq \varepsilon \quad \mbox{ for } h(t)-2\delta\leq x\leq h(t).
\end{equation}

For this $\delta$ there exists a large time $T_1$ such that
$$
|h(t)-h_\infty |\leq \delta, \ \ |g(t)-g_\infty| \leq \delta , \quad  t>T_1.
$$
Then, for any $t>T_1$ and
$x\in [h_\infty -\delta, h(t)] \subset [h_\infty -\delta, h_\infty +\delta]$ we have
\begin{equation}\label{right bdry 1}
u(t,x) \leq U(t,x) \leq C \big[2M(h_\infty +\delta -x)-M^2 (h_\infty +\delta -x)^2\big]
\leq 4MC\delta \leq \varepsilon.
\end{equation}
Similarly we have
\begin{equation}\label{left bdry 1}
u(t,x) \leq \varepsilon,\quad x\in [g(t), g_\infty +\delta],\ t>T_2,
\end{equation}
for some $T_2\geq T_1$. On the other hand, $\|u(t,\cdot)\|_{L^\infty([g_\infty +\delta, h_\infty-\delta])}
\to 0$ as $t\to \infty$. Hence there exists $T_3 \geq T_2$ such that
$$
\|u(t,\cdot)\|_{L^\infty([g_\infty +\delta, h_\infty-\delta])} \leq \varepsilon,\quad t>T_3.
$$
Combining this inequality with \eqref{right bdry 1} and \eqref{left bdry 1} we have
$\|u(t,\cdot)\|_{L^\infty ([g(t),h(t)])} \leq \varepsilon$ for $t>T_3$. This proves the conclusion.

Finally, in case $T_* <\infty$ and $-\infty<g_* <h_* <\infty$, the conclusion can be
proved in the same way by using the function $U(t,x)$.
\end{proof}

The proof of this lemma also gives the following result.
\begin{lem}\label{lem:s->v}
Assume $u(t,x)\leq C$ for all $0\leq t<T_*$ and $x\in [g(t),h(t)]$.
If $h(t)-g(t)$ shrinks as $t\to T_*$, then $u$ also vanishes as $t\to T_*$.
\end{lem}

\begin{proof}
Construct the function $U$ as above. For any $\varepsilon>0$, there exists
$\delta>0$ such that \eqref{delta skip} holds. For this $\delta$ there exists $T_0$ such that
$0<h(t)-g(t)\leq 2\delta$ for $T_0<t<T_*$, and so \eqref{delta skip} implies that
$$
u(t,x) \leq U(t,x) \leq \varepsilon, \quad x\in [g(t),h(t)] \subset [h(t)-2\delta,h(t)].
$$
This proves $\|u(t,\cdot)\|_{L^\infty ([g(t),h(t)])} \to 0$ as $t\to T_*$.
\end{proof}

\subsection{Necessary condition for vanishing}

\begin{lem}\label{lem:v->h<infty}
Assume $u(t,x)\leq C$ for all $0\leq t<T_*$ and $x\in [g(t),h(t)]$.
If $u$ vanishes as $t\to T_*$, then $h_* <\infty$ and $g_*>-\infty$.
\end{lem}

\begin{proof}
When $T_* <\infty$, the conclusions follow from Lemma \ref{lem:bound-general}, so
we only consider the case $T_* =\infty$.

By the definition of vanishing in section 1, $u$ vanishes as $t\to \infty$ if

(i) $g_\infty<h_\infty$ and $u\to 0$ locally uniformly in $(g_\infty,h_\infty)$, or

(ii) $g_\infty= h_\infty$ and $\|u(t,\cdot)\|_{L^\infty ([g(t),h(t)])} \to 0$ as $t\to \infty$.

\noindent
By Lemma \ref{lem:monotonicity} we have $g_\infty<h_0$ and $h_\infty >-h_0$. So
case (ii) reduces to the conclusions immediately.

Now we consider case (i). Consider the problem \eqref{ellip-p} and
take its solution $V(x)$ on a short interval $[0, X)$. In the case (i),
$\|u(t,\cdot)\|_{L^\infty ([g(t),h(t)])} \to 0$ as $t\to \infty$ by Lemma \ref{lem:uniform vanish}.
Hence, there exists $T>0$ such that
$$
u(t,x) \leq \rho:= V(X),\quad t>T,\ x\in [g(t),h(t)].
$$
Choose a large $b$, then the function $V(-x+b)$ is an upper solution of
problem \eqref{p} and it blocks the extension of $h(t)$. Therefore, $h(t)<b$
and $h_\infty <\infty$. $g_\infty >-\infty$ is proved similarly.
\end{proof}

\begin{lem}\label{lem:T_*<infty}
Assume $u(t,x)\leq C$ for all $0\leq t<T_*$ and $x\in [g(t),h(t)]$.
If $u$ vanishes as $t\to T_*$, then $T_*<\infty$ and $h(t)-g(t)\to 0$ as $t\to T_*$.
\end{lem}

\begin{proof}
(i) We first show that $T_* <\infty$.
By Lemma \ref{lem:v->h<infty} we have
$$
h(t) \leq L_0 \mbox{ and } -g(t)\leq L_0 ,\quad t\in [0, T_*)
$$
for some $L_0 >0$. Set $L := 2 (1 + L_0)$ and
$$
\eta_0(x):= \frac{2\varepsilon}{L^2} (L^2 -x^2),
$$
where $\varepsilon>0$ is small such that
$$
8\big( \alpha +\sqrt{\alpha^2 +2K_1} \big) \varepsilon \leq \alpha,\quad 32 \varepsilon \leq \alpha.
$$
Here $K_1:= \max\limits_{0\leq u\leq 1} |f'(u)|$.
Consider the problem
\begin{equation}\label{eta-p}
\left\{
\begin{array}{ll}
 \eta_t = \eta_{xx} + \bar{f}(\eta), &  \bar{g}(t)< x<\bar{h}(t),\ t>0,\\
 \eta(t,\bar{g}(t))= \eta (t,\bar{h}(t))=0 , &  t>0,\\
 \bar{g}'(t)= - \eta_x(t, \bar{g}(t)) +\alpha, & t>0,\\
 \bar{h}'(t) = - \eta_x (t, \bar{h}(t)) - \alpha, & t>0,\\
 -\bar{g}(0)=\bar{h}(0)=L,\ \ \eta(0,x) =\eta_0 (x),& -L\leq x \leq L.
\end{array}
\right.
\end{equation}
where
$$
\bar{f}(\eta) := 2K_1 \eta \Big( 1- \frac{\eta}{2\varepsilon} \Big) \ \ \ \
(\ \geq f(\eta)\ \mbox{ for } 0\leq \eta \leq \varepsilon\ ).
$$
By the definitions of $\bar{f}$ and $\eta_0$, we see that $\eta(t,x)\leq 2\varepsilon$
for all $t\geq 0$. Constructing a function
$$
U^\varepsilon (t,x):= 2\varepsilon [2M(\bar{h}(t)-x) - M^2 (\bar{h}(t) -x)^2]
$$
over $Q := \{(t,x): t>0, \max\{ \bar{g}(t), \bar{h}(t)- M^{-1}\} \leq x \leq \bar{h}(t)\}$, where
$M:= \max \{\alpha +\sqrt{\alpha^2 +2K_1}, 4\}$.
Then in a similar way as in the proof of Lemma \ref{lem:bound-general}
we see that $U^\varepsilon (t,x)$ is an upper solution of \eqref{eta-p} over $Q$ and so
$$
- \eta_x(t, \bar{h}(t))\leq - U^\varepsilon_x (t,\bar{h}(t)) = 4M\varepsilon \leq \frac{\alpha}{2}.
$$
Therefore, $\bar{h}'(t)\leq - \frac{\alpha}{2}$.
$\bar{g}'(t)\geq \frac{\alpha}{2}$ since $\eta(t,x)$ is an even function. Thus
$\bar{h}(t)-\bar{g}(t)\to 0$ as $t\to \overline{T}^* \leq \frac{2L}{\alpha}$.

Lemmas \ref{lem:uniform vanish} and \ref{lem:v->h<infty} imply that, for some $T\in (0,T_*)$,
$u(t,x)\leq \varepsilon$ for all $x\in [g(t),h(t)]$ and $t>T$.
Clearly $\eta_0(x) \geq u(T,x)$ for $x\in [g(T),h(T)]$. By comparison principle we have
$h(t+T)-g(t+T)\leq \bar{h}(t)-\bar{g}(t)$ for $t>0$, and so $T_*$ can not be $\infty$.

(ii) Next we prove that $h_* -g_*>0$ is impossible.
Otherwise, we may assume without loss of generality that
$$
g(t)<-d <d< h(t) \quad \mbox{ for all } t\in [0,T_*).
$$
Choose $\lambda >0$ small such that
$$
\zeta_0(x) := \lambda \cos \frac{\pi x}{2d} \leq u_0(x),\quad x\in [-d,d].
$$
Consider the problem
$$
\left\{
 \begin{array}{ll}
 \zeta_t =\zeta_{xx} - K_\lambda \zeta, & -d<x<d,\ t>0,\\
 \zeta(t,\pm d) =0,& t>0,\\
 \zeta(0,x)=\zeta_0(x), & -d\leq x\leq d,
 \end{array}
 \right.
$$
where $K_\lambda >0$ is a constant satisfying
$$
f(u) \geq -K_\lambda u , \quad 0\leq u\leq \lambda.
$$
Taking $\lambda_1 := K_\lambda + \frac{\pi^2}{4 d^2}$, then by comparison principle we have
$$
u(t,x) \geq \zeta(t,x) \equiv \lambda e^{-\lambda_1 t} \cos \frac{\pi x}{2 d},\quad
t>0,\ x\in [-d,d].
$$
Therefore, $u$ can not vanish in finite time, contradicts the conclusion in (i).
\end{proof}

\begin{remark}\label{remark:not (a)}
By the proof of Lemma \ref{lem:T_*<infty} we see that case (a) in the
definition of {\it vanishing} in section 1 indeed does not occur.
\end{remark}

\noindent
{\it Proof of Theorem \ref{thm:vanishing}}. Theorem 1.1 follows from
Theorem \ref{thm:global}, Lemma \ref{lem:s->v} and Lemma \ref{lem:T_*<infty}.
\hfill $\Box$

\section{Proof of Theorem \ref{thm:convergence}}

Following the ideas of \cite{DM, DuLou} with suitable variations
we can prove the following claims:

\smallskip \noindent
{\it Claim} 1:  The $\omega$-limit set $\omega(u)$ of the solution $u$ consists of solutions of
\begin{equation}\label{stationary}
v_{xx}+f(v)=0,\ \ x\in I_\infty.
\end{equation}

\smallskip \noindent
{\it Claim} 2: $I_\infty$ is a finite interval only if \eqref{ellip-p} has
solution $V_\alpha(x)$ as in Lemma \ref{lem:stationary-alpha} (i), and in this case
$\omega(u)=\{V_\alpha (x -g_\infty)\}$.

\smallskip \noindent
{\it Claim} 3: If $I_\infty=\R^1$, then $\omega(u)$ is either a
constant or $\omega(u)=\{V(\cdot+\mu): \mu\in
[\mu_1,\mu_2]\}$ for some interval $ [\mu_1,\mu_2]\subset [-h_0,h_0]$,
where $V$ is an evenly decreasing positive solution of \eqref{stationary}.

\smallskip
\noindent {\it Claim} 4: If $\omega(u)=\{V(\cdot+\mu):
\mu\in [\mu_1,\mu_2]\}$ for some interval
$[\mu_1,\mu_2]\subset [-h_0,h_0]$, then there exists a
continuous function $\gamma: [0,\infty)\to [-h_0,h_0]$ such that
\[
u(t,x)-V(x+\gamma(t))\to 0 \mbox{ as } t\to\infty \mbox{ locally
uniformly in } \R^1.
\]

Clearly, the conclusions of Theorem \ref{thm:convergence} follow
from these claims. \hfill $\Box$

\bigskip
\noindent
{\it Proof of Remark \ref{V_infty>B}}.\ \
By Claims 3 and 4 in the above proof, $\omega(u)=\{V\}$,
or \eqref{to V(gamma)} holds when $I_\infty =\R^1$.
We remark that in both cases,
$$
S:= \{ \bar{v}: \alpha^2 = 2F(\bar{v} )\} \not= \emptyset \quad \mbox{ and } \quad
V_\infty := \lim\limits_{x\to \infty} V(x) \geq B := \min S.
$$
Otherwise, $S=\emptyset$ or $S \not= \emptyset$ and $V_\infty <B$.
This indicates by Lemma \ref{lem:stationary-alpha} the solution $V^*(x)$
of \eqref{ellip-p} is defined on $[0,X]$ and $V^*(X)> V_\infty$.
Therefore $V^* (-x+b)$ for sufficient large $b$ can be an upper solution of \eqref{p}
which blocks the motion of $h(t)$ to goes to $+\infty$. So $I_\infty $
is a finite domain, a contradiction. \hfill $\Box$

\section{Sufficient conditions for vanishing}

In this section we give some sufficient conditions for vanishing, which
answers one question in section 1.

\begin{prop}\label{prop:vanish cond}
Vanishing happens as $t\to T_*$ if \eqref{cond3} and one of the following conditions hold.
  \begin{itemize}
  \item[(i)] $\alpha^2 > \alpha^2_0 := 2\sup_{v>0} F(v)$;
  \item[(ii)] $\alpha^2 =\alpha^2_0 > 2F(v)$ for all $v>0$;
  \item[(iii)] $\alpha^2 =\alpha^2_0 =2F(\bar{v})$ for some $\bar{v}>0$ and $u_0 (x) \leq
                \widetilde{V}_\alpha (x+b)$ for some $b\in \R$;
  \item[(iv)] $\alpha <\alpha_0$, $u_0 \leq V_\alpha (x +\ell)$ and
   $u_0(x)\not\equiv V_\alpha (x + \ell)$, where $V_\alpha (x)$ is the
  stationary solution of \eqref{ellip-p} with compact support $[0,2 \ell]$.
  \end{itemize}
Vanishing also happens if $h_0$ is sufficiently small and if $u$ is bounded.
\end{prop}

\begin{proof}

(i) If $\alpha >\alpha_0$, then the problem \eqref{ellip-p} has solution $\widehat{V}_\alpha(x)$.
Choose $b>0$ large such that
$$
u_0 (x) \leq \widehat{V}_\alpha (-x+b)  \mbox{ on their common existence interval}.
$$
Then $u(t,x) \leq \widehat{V}_\alpha (-x+b)$ on their common existence interval and so
$h(t)$ is blocked by $b$ and can not moves beyond $b$.
Therefore, $u$ converges to $0$ or a nontrivial solution of  $v''+f(v)=0$
with compact support. The latter is impossible in case $\alpha >\alpha_0$.
Hence $u$ vanishes.

(ii) and (iii) are proved in a similar way as (i).

(iv) By the strong comparison principle we have
\begin{equation}\label{u<V}
u(1,x)< V_\alpha(x+ \ell ) \quad \mbox{for all } x\in [g(1),h(1)].
\end{equation}
\eqref{u<V} implies that, there exists $\epsilon_0 >0$ small, such that
$$
u(1,x)< V_\alpha(x+ \ell +\epsilon) \quad \mbox{for all } x\in [g(1),h(1)],\
\epsilon \in [0,\epsilon_0].
$$
By the convergence result (Theorem \ref{thm:convergence}), if $u$ does not vanish
then it converges to $V_\alpha(x+ \ell)$.
Hence $\lim_{t\to \infty} u(t,x) = V_\alpha (x+ \ell) \leq V_\alpha(x+ \ell +\epsilon)$
for all $\epsilon \in [0,\epsilon_0]$, a contradiction.

Finally we prove that vanishing happens when $h_0$ is sufficiently small and $u$ is bounded.
Assume $u\leq C$. Then we can define a new function $\tilde{f}(u)$ as in the proof of Lemma
\ref{lem:T_*<infty} such that
it is of monostable type, it is bigger than $f(u)$ for $u\in [0,C]$ and
it decreases sufficiently fast for large $u$.
Then consider the solution $\tilde{u} $  of problem \eqref{p} with $\alpha =0$
for any initial data with support $[-h_0, h_0]$.
By Proposition 5.4 in \cite{DuLou} we know that when $h_0$ is sufficiently small,
$\tilde{u}$ vanishes. This $\tilde{u}$ is an upper solution of our problem \eqref{p}.
So the solution $u$ of \eqref{p} also vanishes.
\end{proof}

In case $f$ is of (f$_M$) or (f$_B$) type, we have some further sufficient
conditions for vanishing.

\begin{prop}\label{prop:vanishing}
Let $h_0 >0$ and $\phi \in \mathscr{X}(h_0)$. Then $u$ vanishes
if one of the following conditions holds:

\begin{itemize}
\item [\rm (i)] $f$ is of {\rm (f$_M$)} type and $\|\phi\|_{L^\infty}$ is
sufficiently  small;

\item [\rm (ii)] $f$ is of {\rm (f$_B$)} type, $\|\phi \|_{L^\infty} \leq \theta$, or
$\| \phi\|_{L^1 ([-h_0,h_0])}  \leq \theta \cdot \sqrt{ \frac{2\pi}{eK}}$.
\end{itemize}
\end{prop}

\begin{proof}
Consider the problem \eqref{p} with $\alpha =0$, denote its solution by $\tilde{u}(t,x)$.
By comparison principle we easily have $u\leq \tilde{u}$, that is, the solution $\tilde{u}$ is an 
upper solution of \eqref{p}. So the conclusions of (ii), as well as the conclusions of (i) 
in case $h_0 <\pi/(2\sqrt{f'(0)})$, follow from \cite[Theorem 3.2]{DuLou} immediately.

Now we prove (i) for any $h_0>0$. Since $f$ is of (f$_M$) type, there exists
$K>0$ such that $f(u)\leq Ku\ (u\geq 0)$. Choose $C>0$ such that
\begin{equation}\label{def-C}
2 (\alpha +\sqrt{\alpha^2 +2K} ) C\leq  \alpha,\quad 3C \leq 1.
\end{equation}
For this $C$, we take $\varepsilon >0$ sufficiently small such that
\begin{equation}\label{choose epsilon}
\varepsilon <  \frac{2h_0 \alpha} {\pi C},\quad 16 \varepsilon^2 \Big(1+\frac{\pi}{2h_0}\Big)
\leq 3 \alpha.
\end{equation}
Now we consider the problem
\begin{equation}\label{eta-p-1}
\left\{
\begin{array}{ll}
 \eta_t = \eta_{xx} + K\eta, &  -h_0 < x< h_0,\ t>0,\\
 \eta (t,\pm h_0)= 0 , &  t>0,\\
  \eta (0,x) = \tilde{\phi}(x),& -h_0\leq x \leq h_0,
\end{array}
\right.
\end{equation}
where $\tilde{\phi}(x):=\varepsilon^2 \cos \frac{\pi x}{2h_0}$. Clearly
$\|\tilde{\phi}\|_{C^1 ([-h_0,h_0])}\leq 3\alpha /16$ by the choice of $\varepsilon$. The
solution of \eqref{eta-p-1} is
$$
\eta(t,x) = \varepsilon^2 e^{ \big(K-\frac{\pi^2}{4h^2_0} \big)t} \cos \frac{\pi x}{2h_0}.
$$
Set
$$
T:= \frac{1}{K} \log \frac{C}{\varepsilon} > \frac{2h_0}{\alpha}.
$$
The last inequality follows form the choice of $\varepsilon$.
Denote the solution of \eqref{p} with initial data $u_0(x)=\phi(x)$
by $u(t,x)$.  Since
$$
-\eta_x(t,h_0) \leq \varepsilon^2 e^{KT} \frac{\pi }{2h_0} =
\frac{\pi C \varepsilon}{2h_0} <\alpha\quad \mbox{ for } 0\leq t \leq T,
$$
by the choice of $\varepsilon$, $\eta$ is an upper solution of \eqref{p} and
$$
u(t,x) \leq \eta(t,x) \leq \eta(t,0) \leq \varepsilon^2 e^{Kt} \leq C \quad
\mbox{ for }0\leq t\leq T.
$$

Now we construct function
$$
U(t,x):= C[2M (h(t)-x ) -M^2 (h(t)-x)^2]
$$
over $Q:= \{(t,x): 0\leq t<T, \max\{ h(t)-M^{-1}, g(t)\} \leq x\leq h(t)\}$ as above,
where
$$
M:= \max \Big\{ \frac{\alpha +\sqrt{\alpha^2 +2K}}{2},\ \frac{4\|\tilde{\phi}(x)\|_{C^1}}{3C}\Big\}.
$$
A similar discussion as in the previous sections shows that $u(t,x)\leq U(t,x)$ in $Q$
and so
$$
-u_x(t,h(t)) \leq - U_x(t,h(t)) = 2MC \leq \frac{\alpha}{2}
$$
by the choice of $C$. Therefore,
$$
h'(t) =-u_x(t,h(t))-\alpha \leq -\frac{\alpha}{2},
$$
$g'(t)\geq \frac{\alpha}{2}$ is proved similarly,  so
$$
h(t)-g(t) \leq 2h_0 - \alpha t \to  0\quad \mbox{as } t\to \frac{2h_0}{\alpha} <T.
$$
Therefore, shrinking happens for $(u,g,h)$ in finite time and so
vanishing happens by Theorem \ref{thm:vanishing}.
Finally any solution of \eqref{p} with initial data less than
$\phi$ also vanishes in finite time.
\end{proof}

\section{Proof of Theorem \ref{thm:mono-bi}}
The first half of the theorem, that is, trichotomy result (spreading, vanishing or
transition) follow from Theorems \ref{thm:vanishing} and \ref{thm:convergence} immediately.

Now we prove the second half of the theorem.
The proof is similar as those in \cite[Theorems 5.2 and 5.6]{DuLou}. For
the readers' convenience we give the details below.
By Proposition \ref{prop:vanishing},
the solution $u(t,x;\sigma \phi)$ of \eqref{p} with initial data $\sigma \phi$
vanishes provided $\sigma>0$ is small. Therefore
\[
\sigma^*=\sigma^*(h_0,\phi):=\sup \big\{\sigma_0: u(t,x;\sigma \phi) \mbox{ vanishes
for } \sigma\in (0,\sigma_0]\big\}\in(0,+\infty].
\]
If $\sigma^*=+\infty$, then there is nothing left to prove. So we
assume that $\sigma^*$ is a finite positive number.

By definition, vanishing happens for all $\sigma\in (0,\sigma^*)$.
We now consider the case $\sigma=\sigma^*$. In this case, we cannot
have vanishing, for otherwise we have, for some large $t_0>0$,
$$
u(t_0,x)<\tilde{\phi}(x):= \varepsilon^2 \cos \frac{\pi (x-b)}{2},\quad
x\in [g(t_0), h(t_0)]
$$
for some $b\in \R$, where $\varepsilon$ is chosen as in the proof of
Proposition \ref{prop:vanishing} in (f$_M$) case, and $\varepsilon^2 < \theta$
in (f$_B$) case.  Due to the continuous
dependence of the solution on the initial values, we can find
$\epsilon>0$ sufficiently small such that the solution $(u_\epsilon,
g_\epsilon, h_\epsilon)$ of \eqref{p} with $u_0=(\sigma^*+\epsilon)\phi$ satisfies
\[
u_\epsilon (t_0,x)< \tilde{\phi}(x),\quad x\in [g_\epsilon (t_0), h_\epsilon(t_0)].
\]
Hence we can apply Proposition \ref{prop:vanishing} and its proof to
conclude that vanishing happens for $(u_\epsilon, g_\epsilon,
h_\epsilon)$, a contradiction to the definition of $\sigma^*$. Thus
at $\sigma=\sigma^*$ either spreading or transition happens.

We show next that spreading cannot happen at $\sigma=\sigma^*$.
Suppose this happens. Let $V_\alpha $ be the solution of \eqref{ellip-p}.
Then we can find $t_0>0$ large such that
\begin{equation}\label{u-V_alpha}
 [-\ell  ,\ell]\subset (g(t_0), h(t_0)),\; u(t_0, x)> V_\alpha (x-\ell) \mbox{ in
} [-\ell , \ell ].
\end{equation}
By the continuous dependence of the solution on initial values, we
can find a small $\epsilon>0$ such that the solution $(u^\epsilon,
g^\epsilon, h^\epsilon)$ of \eqref{p} with
$u_0=(\sigma^*-\epsilon)\phi$ satisfies \eqref{u-V_alpha}. Hence spreading happens for $(u^\epsilon,
g^\epsilon, h^\epsilon)$. But this is a contradiction to the
definition of $\sigma^*$.

Hence transition must happen when $\sigma=\sigma^*$. We show next
that spreading happens when $\sigma>\sigma^*$. Let $(u,g,h)$ be a
solution of \eqref{p} with some $\sigma>\sigma^*$, and denote the
solution of \eqref{p} with $\sigma=\sigma^*$ by $(u^*,g^*,h^*)$. By
the comparison theorem we know that
\[
[g^*(1), h^*(1)]\subset (g(1), h(1)),\; u^*(1,x)<u(1,x) \mbox{ in }
[g^*(1),h^*(1)].
\]
Hence we can find $\epsilon_0>0$ small such that for all
$\epsilon\in [0,\epsilon_0]$,
\[
[g^*(1)-\epsilon, h^*(1)-\epsilon]\subset (g(1), h(1)),\;
u^*(1,x+\epsilon)<u(1,x) \mbox{ in }
[g^*(1)-\epsilon,h^*(1)-\epsilon].
\]
Now define
\[
\tilde{u}_\epsilon(t,x)=u^*(t+1, x+\epsilon),\;
\tilde{g}_\epsilon(t)=g^*(t+1)-\epsilon,\; \tilde{h}_\epsilon(t)=h^*(t+1)-\epsilon.
\]
Clearly $(\tilde{u}_\epsilon, \tilde{g}_\epsilon, \tilde{h}_\epsilon)$ is a solution of
\eqref{p} with $u_0(x)=u^*(1, x+\epsilon)$. By the comparison
principle we have, for all $t>0$ and $\epsilon\in (0,\epsilon_0]$,
\[
[\tilde{g}_\epsilon(t), \tilde{h}_\epsilon(t)]\subset (g(t+1), h(t+1)), \;
\tilde{u}_\epsilon(t,x)\leq u(t+1,x) \mbox{ in } [\tilde{g}_\epsilon(t),
\tilde{h}_\epsilon(t)].
\]
If $\omega(u^*)=\{V_\alpha (\cdot-b^*)\}$ and $\omega(u)=\{V_\alpha(\cdot-b)\}$
for some fixed $b^*,b\in \R$, then
it follows from the above inequalities that
\[
V_\alpha (x-b^*+\epsilon)\leq V_\alpha (x-b) \mbox{ for all $x\in [b^*-\epsilon, 2\ell+b^*-\epsilon]$ and
$\epsilon\in (0,\epsilon_0]$.}
\]
By the definition of $V_\alpha (x)$ we have $b^*-\epsilon =b$ for all
$\epsilon \in (0,\epsilon_0]$. This is impossible and so $\omega(u) =\{1\}$,
that is, $u$ spreads. \hfill $\Box$

\begin{remark}
From Propositions 5.4 and 5.8 in \cite{DuLou} we know that if $-f(u)$ grows very fast,
then there exists $\tilde{h}_0$ small such that, the solution of \eqref{p} with
$\alpha=0$ vanishes no matter how large the initial data is, provided $h_0 \leq \tilde{h}_0$.
Since such solutions are upper solution to our problem \eqref{p} ($\alpha>0$),
we know that when $-f$ grows very fast and when $h_0$ is sufficiently small,
we have $\sigma^*(h_0,\phi)=\infty$.
\end{remark}

\bigskip

{\bf Acknowledgement.} The authors would like to thank Professors Y. Du and Z. Lin
for valuable discussion on the free boundary conditions.

\end{document}